\documentclass[11pt]{amsart}

\usepackage{ae,amsfonts,euscript,enumerate}

\usepackage{amsmath,euscript}

\usepackage{amssymb}

\usepackage{amsthm}

\usepackage{enumerate}

\usepackage{amsfonts}

\usepackage{comment}

\usepackage{colortbl}

\usepackage{amssymb,amsmath,array}

\numberwithin{equation}{section}

 \theoremstyle{plain}

\newtheorem{thm}[equation]{Theorem}

\newtheorem{lem}[equation]{Lemma}

\newtheorem{prop}[equation]{Proposition}

\newtheorem{cor}[equation]{Corollary}

\newtheorem{hypothesis}[equation]{Hypothesis}

\def\co{{\mathcal O}}

\def\oqmm13{\co_q(M_{1,3})}

\def\oqm23{\co_q(M_{2,3})}

\newcommand{\mb}{\mathbb}

\usepackage{amssymb}

\usepackage{amsmath}

\usepackage[latin1]{inputenc}

\usepackage{amsmath,euscript}

\usepackage{amssymb}

\usepackage{amsthm}

\usepackage{enumerate}

\usepackage{amsfonts}

\usepackage{comment}

\usepackage{colortbl}

\usepackage{a4wide}

\usepackage{amssymb,amsmath,array}

\title[]{Free subalgebras of division algebras over uncountable fields}

\author{Jason P.~Bell}
\thanks{The first-named author was supported by NSERC grant 31-611456.}

\subjclass[2000]{}

\address{Jason Bell\\
Department of Pure Mathematics\\
University of Waterloo\\
Waterloo, ON N2L 3G1\\
Canada}

\email{jpbell@uwaterloo.ca}

\author{D. Rogalski}
\thanks{The second-named author was supported by NSF Grant DMS-0900981.}


\address{D. Rogalski\\
Department of Mathematics\\
University of California, San Diego\\
La Jolla, CA 92093-0112\\
USA}

\email{drogalsk@math.ucsd.edu}

\begin{document}

\bibliographystyle{plain}


\begin{abstract} We study the existence of free subalgebras in division algebras, and prove the following general result:  if $A$ is a noetherian domain which is countably generated over an uncountable algebraically closed field $k$ of characteristic $0$, then either the quotient division algebra of $A$ contains a free algebra on two generators, or it is left algebraic over every maximal subfield.  As an application, we prove that if $k$ is an uncountable algebraically closed field and $A$ is a finitely generated $k$-algebra that is a domain of GK-dimension strictly less than $3$, then either $A$ satisfies a polynomial identity, or the quotient division algebra of $A$ contains a free $k$-algebra on two generators.  \end{abstract}


\maketitle

\tableofcontents


\section{Introduction}

Many authors have noted that division rings and their multiplicative groups often have free subobjects \cite{Ch, FGS, Licht, Lor, ML, ML15, ML2, ML3, RV, SG}. In this paper, we continue the study of the question of when division algebras contain a free subalgebra in two generators over their center.   Makar-Limanov  first gave evidence of this phenomenon by showing that if $A=k\{x,y\}/(xy-yx-1)$ is the Weyl algebra over a field $k$ of characteristic $0$, then its quotient division ring contains a free $k$-algebra on two generators \cite{ML}.   We call the \emph{free subalgebra conjecture} the statement that a division algebra $D$ must contain a free algebra on two generators over its center unless $D$ is locally PI, that is, all of its affine subalgebras are polynomial identity algebras.  This conjecture (in some form) was formulated independently by both Makar-Limanov and Stafford.    We refer the reader to \cite{BR} for a more detailed discussion of the conjecture and past work on the subject.

The authors showed in \cite{BR} that the free subalgebra conjecture holds for the quotient division rings of iterated Ore extensions of PI rings, if the base field $k$ is uncountable (and in most cases likely to arise in practice, over a countable base field $k$ as well.)  In this sequel paper, we develop further techniques for demonstrating the existence of a free subalgebra of a division ring, which are available only in the case of an uncountable base field.  The main advantage of an uncountable base field is that with this hypothesis, we can prove that a division algebra $D$ contains a free subalgebra on $2$ generators if and only if $D(t)$ does, where $t$ is a commutative indeterminate.  Using a criterion for existence of free subalgebras proved in \cite{BR}, in this paper we prove the following much stronger criterion:
\begin{thm}  {\upshape (Corollary~\ref{cor: xxx2}) }
Let $D$ be a division algebra over an uncountable field $k$.  Let $a \in D$ be nonzero and let $E := C(a; D)$ be the centralizer of $a$ in $D$.  If $\operatorname{char} k = p > 0$ then assume further that there is no $u \in D$ such that $aua^{-1} = u + 1$.  If $D$ is neither left nor right algebraic over $E$, then $D$ contains a free $k$-algebra on two generators.
\end{thm}

The notion of algebraicity used in the theorem above is defined as follows:   If $D$ is a division ring and $E$ is a division subring (not necessarily central), then $D$ is \emph{left algebraic} over $E$ if every $x$ in $D$ satisfies a non-trivial polynomial equation $\sum_{i=0}^m c_i x^i = 0$ with $c_0,\ldots ,c_m\in E$.  \emph{Right algebraic} is defined analogously.  Thus another way of phrasing the result in characteristic $0$ is the following:  if $D$ does not contain a free subalgebra, then $D$ must be algebraic (on at least one side) over every centralizer.  Intuitively, having algebraic-like properties is a definite obstacle to containing a free subalgebra, and the theorem shows that in some sense this is the only obstacle.

When $D$ is the quotient division ring of a noetherian domain over an uncountable algebraically closed field, the criterion above can be tightened even further, as follows.
\begin{thm} {\upshape (Corollary~\ref{cor-mainthm2})}
Let $A$ be a countably generated noetherian domain over an uncountable, algebraically closed field $k$ of characteristic $0$.  If the quotient division ring $D := Q(A)$ of $A$ does not contain a free $k$-algebra on two generators, then $D$ is left algebraic over every maximal subfield.
\end{thm}

The property of being algebraic over every maximal subfield seems rather similar to the property of being algebraic over the center, and in fact no examples are known of either phenomenon except for division rings which are locally PI.

The question of the existence of finitely generated division algebras which are algebraic but not finite-dimensional over their centers is the famous Kurosh problem for division rings.  Thus the theorem above more or less settles the free subalgebra conjecture for $\mb{C}$-algebras, modulo the question of the potential existence of bizarre examples satisfying a Kurosh-type condition.  In practice, of course, an explicit finitely generated division ring is likely to be obviously either PI or else provably not algebraic over a maximal subfield.

In the remainder of the paper, we apply our criterion to the special case of domains of low Gelfand-Kirillov (GK)-dimension, and prove the following result.
\begin{thm} {\upshape (Corollary~\ref{cor-mainthm3})}
Let $k$ be an uncountable algebraically closed field, and let $A$ be a finitely generated $k$-algebra that is a domain of GK-dimension strictly less than $3$. If $A$ does not satisfy a polynomial identity, then the quotient division algebra $Q(A)$ of $A$ contains a free $k$-algebra on two generators.  Thus $Q(A)$ satisfies the free subalgebra conjecture.  \label{thm: GK2}
\end{thm}

The result above may be seen as a generalization of Makar-Limanov's original result about free subalgebras in the quotient division ring of the complex Weyl algebra, since the Weyl algebra has GK-dimension 2.  We actually show that Theorem \ref{thm: GK2} holds for non-algebraically closed base fields $k$ as well, as long as $A$ is not an algebraic $k$-algebra.  Notice that if $A$ is a finitely generated algebraic $k$-algebra which is a domain, it is its own quotient division ring.  Currently, it is even unknown whether or not there exists a division algebra which is finitely generated as an algebra but infinite-dimensional over its center.  This is another problem in the same general family as the Kurosh problem, and such problems are among the most difficult in ring theory.  For this reason we do not expect the methods used in this paper will suffice to extend Theorem~\ref{thm: GK2} to the case of a non-algebraically closed or a countable base field.

\section*{Acknowledgments}

We thank George Bergman, James Zhang, Sue Sierra, Tom Lenagan, Agata Smoktunowicz, Toby Stafford, Zinovy Reichstein, Lance Small, and Jairo Gon\c calves  for valuable discussions and helpful comments.

\section{Free subalgebras of a division algebra and indeterminate extensions}
\label{sec-crit}

In this section, we prove a proposition that shows that over an uncountable field, replacing a division algebra $D$ by $D(t)$ does not affect the existence of free subalgebras.  This surprisingly useful freedom to add an indeterminate is the key to the proofs of the main results of this paper.  We work with right rings of fractions and so by Ore set we mean a right Ore set, and similarly by Ore domain we mean a right Ore domain.

Given a ring $R$ which is an Ore domain, we denote its quotient division ring by $Q(R)$. Given any Ore domain $R$ and commutative indeterminate $t$,  $R[t]$ is also an Ore domain and we use the notation $Q(R)(t)$ for its quotient division ring.   More generally, given a ring $R$ with automorphism $\sigma: R \to R$, recall that $R[x; \sigma]$ is the ring $\bigoplus_{n \geq 0} Rx^n$ with multiplication rule $ax^n bx^m = a\sigma^n(b) x^{n+m}$.  If $R$ is an Ore domain, then so is $R[x; \sigma]$ and in this case its Ore ring of fractions is denoted $Q(R)(x; \sigma)$.

We now state the main result of this section.   The use of the uncountable base field hypothesis is similar to other results of this type.  In particular, Reichstein used similar methods to show in \cite{R} that an algebra $R$ over an uncountable base field $k$ contains a free subalgebra on two generators if and only if any base field extension $R \otimes_k K$ does.
\begin{prop}
\label{prop-extension}
\label{lem: Ro}
Let $k$ be an uncountable field and let $D$ be a division algebra over $k$. Then the following are equivalent:
\begin{enumerate}
\item $D$ contains a free $k$-algebra on two generators;
\item $D(t)$ contains a free $k(t)$-algebra on two generators;
\item $D(t)$ contains a free $k$-algebra on two generators.
\end{enumerate}
\end{prop}

Before giving the proof of the proposition, we start with a few subsidiary lemmas.  The following
straightforward result is unpublished work of the first-named author and John Farina; since the proof is short we include it here.
\begin{lem}
\label{ore-lem}
Let $k$ be a field and let $R$ be an Ore domain which is a countably generated $k$-algebra.  Then any countable subset $X$ of $R$ is contained in a countable Ore set $S$ in $R$.
\end{lem}
\begin{proof}
Let $X := X_1$.  We construct inductively a sequence of countable subsets $X_1 \subseteq X_2 \subseteq \cdots$ with the property that if $x \in X_n$ and $r \in R$, then there is $y \in X_{n+1}$ and $t \in R$ such that $ry = xt$.  Then $S := \bigcup_{n \geq 1} X_n$ is the required countable Ore set.

Let $\{ r_1, r_2, \dots \}$ be a countable $k$-basis for $R$.  Suppose that $X_n$ has been constructed and write $X_n = \{ x_1, x_2, \dots \}$.  For each $i, j \in \mb{N}$ there are nonzero elements $t, u \in R$ such that $r_i t = x_j u$.  Let $T$ be the set of all $t$'s occurring as $i$ and $j$ vary.  Every finite subset of $T$ has a nonzero common right multiple $a$, and we let $X_{n+1}$ be a set containing one such common multiple for each finite subset of $T$.

Now given an arbitrary element of $R$, say $r = \sum \alpha_i r_i$ with $\alpha_i \in k$, and $x \in X_n$, we have $t_i \in T, u_i \in R$ such that $r_i t_i = x u_i$.  Then there is $a \in X_{n+1}$ and $b_i \in R$ such that $a = t_i b_i $ for each $i$, so that $r_i a = r_i t_i b_i = x u_i b_i$.  Then $ ra = (\sum \alpha_i r_i) a =  x (\sum \alpha_i u_i b_i)$ and so $X_{n+1}$ has the claimed property.
\end{proof}

\begin{lem}
\label{lift-lem}
Let $R$ be a $k$-algebra which is an Ore domain.   Let $f_1(t) g_1(t)^{-1}, \dots, f_n(t) g_n(t)^{-1}$, where $f_i(t), 0 \neq g_i(t) \in R[t]$, be a finite set of elements of $Q(R)(t)$.  Suppose there are infinitely many $\alpha \in k$ such that $g_i(\alpha) \neq 0$ for all $i$ and the set $\{ f_i(\alpha) g_i(\alpha)^{-1} \}$ is dependent over $k$.  Then the set $\{ f_i(t) g_i(t)^{-1} \}$ is dependent over $k(t)$.
\end{lem}
\begin{proof}
For any $\alpha \in k$ there is an evaluation homomorphism $R[t] \to R$ defined by $t \mapsto \alpha$.  Note that a nonzero polynomial in $R[t]$ has only finitely many zeroes in $k$.  Choose a nonzero common right multiple $h$ for the $g_i$, say $h = g_i b_i$, so that $f_i g_i^{-1} = f_i b_i h^{-1}$.  Since $b_i(\alpha) = 0$ for finitely many $\alpha$, there are still infinitely many $\alpha \in k$ such that $h_i(\alpha) \neq 0$ for all $i$ and the elements $f_i(\alpha) b_i(\alpha) h(\alpha)^{-1}$ are $k$-dependent.  Then there are infinitely many $\alpha \in k$ such that the elements $f_i(\alpha) b_i(\alpha)$ are $k$-dependent.

Let $d_i(t) := f_i(t) b_i(t)$.  Choose a finite-dimensional $k$-space $V \subseteq R$ such that all coefficients of the polynomials $d_i$ are in $V$.  Let $\{v_1, \dots, v_m \}$ be a $k$-basis of $V$.  Write $d_i(t) = \sum_j a_{ij}(t) v_j$ with $a_{ij}(t) \in k[t]$.  Then $\{ d_i(\alpha) \}$ is $k$-dependent if and only if all $n \times n$-minors of the $n \times m$ matrix $( a_{ij}(\alpha) )$ are $0$; equivalently, if and only if all $n \times n$ minors of the matrix $(a_{ij}(t) )$ vanish at $\alpha$.  Since each minor vanishes at either finitely many $\alpha \in k$ or else is identically $0$, the assumption that the $d_i(\alpha)$ are $k$-dependent for infinitely many $\alpha$ implies that all $n \times n$ minors of $(a_{ij}(t) )$ are $0$.  But this means that the $d_i(t)$ are dependent over the field $k(t)$.  Then the elements $f_i b_i h^{-1} = f_i g_i^{-1}$ are linearly dependent over $k(t)$ as required.
\end{proof}

\begin{proof}[Proof of Proposition~\ref{prop-extension}] By a result of Makar-Limanov and Malcolmson \cite[Lemma 1]{ML3}, for any elements $x, y$ in a $k$-algebra, the subalgebra $k \langle x, y \rangle$ is free over $k$ if and only if $k_0 \langle x, y \rangle$ is free over $k_0$, where $k_0$ is the prime subfield of $k$.  It thus follows that if two elements generate a free subalgebra over some central subfield, then they generate a free subalgebra over every central subfield. Thus (2) and (3) are equivalent. Moreover, (1) obviously implies (3) and so we concentrate on proving that (2) implies (1).

Assume that $D(t)$ contains a free $k(t)$-algebra on two generators, generated by $x$ and $y$.   Let $\{ w_i \,| \, i \geq 0\}$ be the set of all words in $x$ and $y$ (so by assumption the $w_i$ are $k(t)$-independent) and write $w_i = f_i(t)g_i(t)^{-1}$ for some $0 \neq g_i, f_i \in D[t]$.  There are countably many coefficients in $D$ among the $f_i$ and $g_i$, so we can find a countably generated $k$-subalgebra $R \subseteq D$ such that $f_i, g_i \in R[t]$ for all $i$.  If $R$ is not an Ore domain, then it is standard that $R$ already contains a free $k$-algebra in 2 generators \cite[Proposition 4.13]{KL}, so $D$ does as well and we are done.  Thus we can assume that $R$ is an Ore domain, and so $R[t]$ is an Ore domain as well which is countably generated over $k$.

Now by Lemma~\ref{ore-lem}, we can choose a countable Ore set $S$ in $R[t]$ which contains all of the $g_i$.  Let $T := R[t]S^{-1}$.  Note that $T$ contains the $w_i$.  Given $\alpha \in k$, the evaluation map $R[t] \to D$ with $t \mapsto \alpha$ extends to a map $\phi_{\alpha}: T \to D$ as long as $s(\alpha) \neq 0$ for all $s \in S$.  There are only countably many $\alpha \in k$ which are a root of some $s \in S$ and hence the evaluation map $\phi_{\alpha}: T \to D$ is well-defined for all $\alpha$ in an an uncountable subset $X \subseteq k$.

Consider the first $n$ words $\{ w_1, \dots, w_n \}$ for some $n$.  Since the set $\{ w_i \,| \, 1 \leq i \leq n \}$ is $k(t)$-independent, by Lemma~\ref{lift-lem} there are at most finitely many $\alpha \in X$ such that the set $\{ f_i(\alpha) g_i(\alpha)^{-1} \,| \, 1 \leq i \leq n \}$ is $k$-dependent.  Taking the union over all $n$,  there are at most countably many $\alpha \in X$ such that $\{ \phi_{\alpha}(w_i) \,| \, i \geq 0 \}$ is $k$-dependent. Thus we can find some $\alpha \in X$ such that $\{ \phi_{\alpha}(w_i) \,| \, i \geq 0 \}$ is $k$-independent, in other words such that $\phi_{\alpha}(x), \phi_{\alpha}(y)$ generate a free $k$-subalgebra of $D$.
\end{proof}

The proposition above already simplifies the proof of an important special case of the main theorem from \cite{BR}, and we describe this next as a demonstration of the power of this method.
We first recall the main criterion for the existence of a free subalgebra from \cite{BR}.  In Section~\ref{sec: main}, we will combine this criterion with Proposition~\ref{prop-extension} to prove a much stronger criterion.
\begin{thm} {\upshape (cf. \cite[Theorem 2.3]{BR}) }
\label{main-criterion-thm}
\label{thm-BR}
Let $D$ be a division algebra over $k$ with $k$-automorphism $\sigma: D \to D$, and let
$E := \{u\in D \,| \, \sigma(u)=u\}$ denote the division subalgebra of $D$ fixed by $\sigma$.  If $k$ has characteristic $p>0$, assume further that there does not exist $u \in D$ such that $\sigma(u) = u + 1$.  If there is $b \in D\setminus E$ such that the equation
\begin{equation}
\sigma(u)-u \ \in \ b+E
\end{equation}
has no solutions with $u \in D$, then the $k$-algebra generated by $b(1-x)^{-1}$ and $(1-x)^{-1}$ is a free $k$-subalgebra of $D(x;\sigma)$.
\end{thm}
\begin{proof}
This is just a restatement of \cite[Theorem 2.3]{BR}, with a different hypothesis in case $k$ has characteristic $p$.  The proof of \cite[Theorem 2.3]{BR} shows that the conclusion of the theorem holds unless there is $u \in D$ such that $\sigma(u) = u + 1$.
\end{proof}

The following special case of \cite[Theorem 1.1]{BR} is now quite easy.
\begin{thm}
\label{thm: BR}
Let $K/k$ be a field extension, where $k$ is an uncountable field of characteristic $0$, and let $\sigma: K \to K$ be a $k$-automorphism.  Then $K(x; \sigma)$ contains a free $k$-subalgebra on 2 generators if and only if $K(x; \sigma)$ is not locally PI, if and only if there exists $a \in K$ such that $\{ \sigma^i(a) \,| \, i \in \mb{Z} \}$ is infinite.
\end{thm}

\begin{proof} Looking at the proof of \cite[Theorem 1.1]{BR}, the only difficult implication in this theorem is to prove that if there exists $a \in K$ such that $\{ \sigma^i(a) \,| \, i \in \mb{Z} \}$ is infinite, then $K(x; \sigma)$ contains a free subalgebra on 2 generators.  By Proposition~\ref{prop-extension}, it is enough to find a free subalgebra on two generators inside $K(x; \sigma)(t) \cong K(t)(x; \sigma)$, where $\sigma$ is extended trivially to an automorphism of $K(t)$.  For this we use the valuative criterion \cite[Lemma 2.4]{BR}, which is an easy extension of Theorem~\ref{main-criterion-thm} above.  That criterion shows that it is enough to find a discrete valuation $\nu$ of $K(t)$ which is nontrivial and such that for any $b \in K(t)$, $\nu(\sigma^i(b)) = 0$ for all $i \gg 0$ and $i \ll 0$.  Taking the discrete valuation $\nu$ on $K(t)$ corresponding to the maximal ideal $(t-a) \subseteq K[t]$, because $a$ is on an infinite $\sigma$-orbit these properties of $\nu$ are immediate.
\end{proof}

\noindent Compared to the proof just given, the methods of \cite{BR} require a more complicated geometric argument to construct a suitable valuation of $K$ itself.  On the other hand, \cite{BR} also proves this theorem in characteristic $p$ and, in many important cases, over a countable base field.

The following result, which is along similar lines to Proposition~\ref{prop-extension}, will be useful later in applying Theorem~\ref{thm-BR} in characteristic $p$.
\begin{lem} Let $k$ be an uncountable field, and let $A$ be a countably generated $k$-algebra which is an Ore domain.  Let $D := Q(A)$.  Suppose we have a $k$-algebra automorphism $\sigma: D \to D$, which we extend to an automorphism $\sigma: D(t) \to D(t)$ with $\sigma(t) = t$.  Then $D(t)$ contains an element $u$ such that $\sigma(u) = u + 1$ if and only if $D$ contains such an element.
\label{lem: weyl}
\end{lem}

\begin{proof} Suppose that $v = f(t) g(t)^{-1} \in D(t)$ satisfies $\sigma(v) = v + 1$, where $f(t), 0 \neq g(t) \in A[t]$.  By Lemma \ref{ore-lem}, there exists a countable Ore set $S$ of $A[t]$ that contains $g(t)$.  Then there are uncountably many $\alpha\in k$ for which $h(\alpha) \neq 0$ for all $h \in S$ and hence the evaluation homomorphism from $A[t]$ to $A$ given by $t \mapsto \alpha$ extends to a homomorphism $\phi: A[t]S^{-1} \to D$.   Note also that for any element $w \in A[t]S^{-1}$ such that $\sigma(w) \in A[t]S^{-1}$, $\phi \sigma(w) = \sigma \phi(w)$.  Thus choosing any $\alpha$ for which $\phi$ is defined, we see that $u := \phi(v)$ satisfies $\sigma(u) = u + 1$.  The other direction is trivial.
\end{proof}

We close this section with an easy observation about centralizers in indeterminate extensions.

\begin{lem}
\label{dumb-lem}
Let $D$ be a division ring.  Then the following assertions hold.
\begin{enumerate}
\item[(1)] Let $\{ \sigma_i \,| \, i \in S \}$ be a collection of automorphisms of $D$.
If we define
\[
E:=\{a\in D \,| \, \sigma_i(a)=a\ \text{for all}\ i \in S \},
\]
then
\[
E(t)=\{a(t)\in D(t) \,| \, \sigma_i(a(t))=a(t)\ \text{for all}\ i \in S \},
\]
where we extend each $\sigma_i$ to $D(t)$ by declaring that
$\sigma_i(t)=t$.
\item[(2)] If $X \subseteq D$ is a subset and $Z := C(X; D)$ is its centralizer in $D$, then $C(X;
D(t)) = Z(t)$.  In particular, if $K$ is a maximal subfield of $D$ then $K(t)$ is a maximal subfield of $D(t)$.
\end{enumerate}
\end{lem}
\begin{proof} We first prove (1).  It is clear that the elements of $E(t)$ are $\sigma_i$-fixed for all $i$.  Thus it is sufficient to show that if $a(t)\in D(t)$ is $\sigma_i$-fixed for all $i$ then $a(t)\in E(t)$.
Also, it is straightforward that the elements of $D[t]$ that are $\sigma_i$-fixed for all $i$ are precisely the elements of $E[t]$.    Let $a(t) \in D(t)$ be $\sigma_i$-fixed for all $i$ and let $$J:=\{f(t) \in D[t] \,| \, a(t)f(t)\ \text{is in}\ D[t]\}.$$  Then $J$ is a nonzero right ideal which is $\sigma_i$-stable for all $i$.  Let $g(t)\in D[t]$ be a nonzero polynomial in $J$ of minimal degree.  Without loss of generality, we may assume that $g(t)$ is monic.  Then for each $i$, $g(t)-\sigma_i(g(t))$ is an element of $J$ which has strictly smaller degree than $g(t)$ and hence must be zero.  It follows that $g(t)\in E[t]$.  But now $f(t):=a(t)g(t)\in D[t]$ is also $\sigma_i$-fixed for all $i$ and hence must be an element of $E[t]$. It follows that $a(t)=f(t)g(t)^{-1}\in E(t)$, establishing (1).

To prove (2), observe that if we take $\sigma_x$ to be the automorphism of $D$ given by $d\mapsto xdx^{-1}$, then $Z=C(X;D)$ is precisely the set of elements of $D$ that are $\sigma_x$-fixed for all $x \in X$.  Thus (i) gives that $C(X;D(t)) = Z(t)$.  The final claim follows since if $K$ is maximal subfield of $D$, then $C(K; D) = K$ and so $K(t) = C(K; D(t))$; thus $K(t) = C(K(t); D(t))$.
\end{proof}

\section{Lemmas on algebraicity}
If $R$ is a ring and $E$ is a subring which is a division algebra, we say that $b \in R$ is \emph{left algebraic} over $E$ if $\sum_{i \geq 0} Eb^i$ is not direct.  Right algebraic is defined similarly.  We say that $R$ is left algebraic over $E$ if every element of $R$ is algebraic over $E$, and similarly on the right.  Of course this notion is not nearly so well-behaved as in the commutative case, since if $b$ is left algebraic over $E$ then $\sum_{i \geq 0} Eb^i$ is merely a finite-dimensional left $E$-vector space and not a subring of $R$ in general.  In this section, we collect some basic results about algebraic elements which we will need in later sections.

Suppose that $E$ is a division subring of $D$.  It is not obvious in general that if $D$ is left algebraic over $E$, it must also be right algebraic over $E$.  However, in Proposition~\ref{prop: fix} below we show this is true in the special case that $E$ is a finitely generated field extension of a central base field $k$.  The proof depends on the following two lemmas.

\begin{lem}

\label{lem: bimod}
Let $L$ be a finitely generated field extension of $k$.  Suppose that $V$ is a $k$-central $(L, L)$-bimodule. Then $\dim_L ({}_L V) < \infty$ if and only if $\dim_L (V_L) < \infty$.
\end{lem}
\begin{proof}
By symmetry, it is enough to show that if  ${\rm dim}_L(_{L}V ) <\infty$ then ${\rm dim}_L(V_L) < \infty$.  Suppose that ${\rm dim}_L(_{L}V ) = n$, and fix a left $L$-basis for $V$. Right multiplication by
$x \in L$ is a left $L$-linear map, and so with respect to the fixed basis it corresponds to a
matrix in $M_n(L)$ (acting on $V \cong L^n$ on the left). This determines a $k$-algebra homomorphism
$\phi : L\to M_n(L)$, and this homomorphism completely determines the right $L$-structure on $V$.

Let $L' := \phi(L) \subseteq M_n(L)$. We claim that $M_n(L)$ is a finite-dimensional $L'$-module on both sides.  To see this, let $d$ denote the transcendence degree of $L$ as an extension of $k$.  Observe that $L'\cong L$ as a $k$-algebra and hence $L'$ also has transcendence degree $d$ over $k$.  Let $C$ be a finitely generated $k$-algebra that generates $L'$ as a field extension of $k$ and let $W$ be a finite-dimensional $k$-vector subspace of $C$ that contains $1$ and generates $C$ as a $k$-algebra.  Similarly, we let $B$ be a finitely generated subalgebra of $M_n(L)$ that contains $C$ and whose artinian ring of quotients is equal to $M_n(L)$.  Also, we let $U$ be a finite-dimensional $k$-vector subspace of $B$ that contains $W$ and which generates $B$ as a $k$-algebra.  
Note that $C$ has GK-dimension $d$, since it is a finitely generated commutative $k$-algebra whose field of fractions has transcendence degree $d$ over $k$.  Then $B$ also has GK-dimension $d$.  This can be seen by noting that $B$ contains $C$, and thus the GK-dimension of $B$ is at least $d$.  On the other hand, $B\subseteq M_n(L)$ and $M_n(L)$ is a finite $L$-module (regarding $L$ as scalar matrices).  Hence the GK-dimension of $B$ is at most the GK-dimension of $L$ as a $k$-algebra, which is the transcendence degree of $L$ over $k$.   

We now let $U_m=U^mL'$ for $m\ge 1$.  We claim that there is some $N$ such that $U_i=U_N$ for all $i\ge N$.  Suppose that this is not the case.   Then we cannot have $U_i = U_{i+1}$ for any $i$, as this would imply $U_i = U_N$ for all $i \ge N$.  
Thus there exist elements $u_i\in U^i$ such that the sum $$\sum u_i L'$$ is direct.  This means that 
$$\sum_{i=1}^m u_i W^m$$ has dimension at least $m \, {\rm dim}_k(W^m)$.  Since 
$ \sum_{i=1}^m u_i W^m \subseteq U^{2m}$, we see that ${\rm GKdim}(B)\ge 1+{\rm GKdim}(C)$, which gives a contradiction, since we have shown that both algebras have GK-dimension $d$.
This proves the claim.  Thus $BL' = \bigcup_{i \geq 0} U^i L'$ is finite-dimensional as a right $L'$ vector space.   Hence there exist elements $w_1,\ldots ,w_r\in B$ such that $BL' = w_1 L'+\cdots + w_rL'$, where the sum on the right-hand side is direct.    We now claim that
$M_n(L)$ is at most $r$-dimensional as a right $L'$-vector space.  To see this claim, note that if 
$x_1,\ldots ,x_{r+1}\in M_n(L)$ are nonzero elements such that $x_1L'+\cdots + x_{r+1}L'$ is direct, then since $M_n(L)$ is the artinian ring of quotients of $B$, there is some regular element $b\in B$ such that $bx_i \in B$ for $i=1,\ldots r+1$ and hence $$\sum_{i=1}^{r+1} (bx_i) L'$$ is direct.  This contradicts the fact that $BL'$ has dimension $r$ as a right $L'$-vector space.  A similar argument shows that $M_n(L)$ is finite-dimensional as a left $L'$-vector space.

Now $V$ is certainly a finitely generated (in fact cyclic) left $M_n(L)$-module. Since $L$ acts on
the right via restricting the left $M_n(L)$-action to $L'$, and $M_n(L)$ is module-finite over $L'$, we
see that $V$ is also a finitely generated right $L$-module, as required.
\end{proof}

\noindent The proof above was inspired by the extensive study of bimodules over fields in \cite{NP}.  We note that the lemma above is not true in general when $L$ is an infinitely generated field extension of $k$; see \cite[Example 2.2]{NP}.

\begin{lem} Let $k$ be a field and let $D$ be a division $k$-algebra.  Suppose that $D$ is left algebraic over a subfield $L$ that contains $k$.   If $V$ is a finite-dimensional left $L$-vector subspace of $D$, and $x\in D$, then $$\sum_{i\ge 0} Vx^i$$ is finite-dimensional as a left $L$-vector space.
\label{lem: V}
\end{lem}
\begin{proof} Write $V = \sum_{k=1}^m L a_k$ for some $a_k \in D$.  We claim that for any $a \in D$, $\sum_{i \geq 0} L a x^i$ is finite-dimensional over $L$ on the left.  To see this, note that $\sum_{i\ge 0} L (ax^i a^{-1}) = \sum_{i \ge 0} L (axa^{-1})^i$ is finite-dimensional over $L$ on the left since $D$ is left algebraic over $L$.  Since right multiplication by $a$ is a left $L$-vector space isomorphism, we see that $\sum_{i\ge 0} Lax^i$ is finite-dimensional over $L$ on the left, proving the claim.

This claim clearly implies the lemma since $\sum_{i \geq 0} Vx^i =  \sum_{k=1}^m \sum_{i \geq 0} L a_k x^i$.
\end{proof}

\begin{prop}
\label{lem: fix}
\label{prop: fix}
Suppose that $D$ is a division $k$-algebra with a subfield $L \supseteq k$ such that $L/k$ is a finitely generated field extension.
\begin{enumerate}
\item If $D$ is algebraic over $L$ on one side, then for any finite-dimensional $k$-subspace $V \subseteq D$, $LVL$ is finite-dimensional over $L$ on both sides.
\item If $D$ is algebraic over $L$ on one side, it is algebraic over $L$ on both sides.
\end{enumerate}
\end{prop}

\begin{proof} (1). By symmetry, we may assume that $D$ is left algebraic over $L$.  Let $x_1,\ldots ,x_m$ be a set that generates $L$ as a field extension of $k$.   We now show that $LVL$ is finite-dimensional as a left $L$-vector space.  If $LVL$ is infinite-dimensional as a left $L$-vector space, then we claim that $$W:=\sum_{i_1,\ldots ,i_m\ge 0} LVx_1^{i_1}\cdots x_m^{i_m}$$ must be infinite-dimensional as a left $L$-vector space.  If not, say $\dim_L W  = d$, then given a set of $d + 1$ left $L$-independent elements of $LVL$, by multiplying by a common right denominator in $k[x_1, \dots, x_n]$ we get a set of $d+1$ left $L$-independent elements of $W$, a contradiction. Now by Lemma \ref{lem: V}, the space $V_1:=\sum_{i\ge 0} LVx_1^i$ is finite-dimensional as a left $L$-vector space.  Applying Lemma \ref{lem: V} again, we see that $V_2:=\sum_{i\ge 0} LV_1x_2^i = \sum_{i,j\ge 0} LVx_1^i x_2^j$ is finite-dimensional as a left $L$-vector space.  Continuing in this manner, we see that $W$ must be finite-dimensional as a left $L$-vector space.  Thus $LVL$ is finite-dimensional on the left.  It must also be finite-dimensional over $L$ on the right by Lemma~\ref{lem: bimod}.

(2).  Again it suffices to assume that $D$ is left algebraic over $L$.  Let $b \in D$.  Then $\sum_{i \geq 0} Lb^i  = \sum_{i=0}^n Lb^i$ for some $n$.  Let $V := \sum_{i=0}^n kb^i$.  Then $LVL$ is finite-dimensional over $L$ on both sides by part (1).  We have $LV = \sum_{i \geq 0} Lb^i$ and so $LVL = \sum_{i \geq 0} Lb^i L$.  Thus $\sum_{i \geq 0} b^i L$ must also be finite-dimensional over $L$ on the right.  Thus $D$ is also right algebraic over $L$.
\end{proof}

\noindent As remarked in the introduction, we know of no example of a division ring which is algebraic over some subfield and not locally PI.  However, we will need to control such potential rings in our proofs later and the result above will be useful for this purpose.  Another annoyance is the following:  if $D$ is a division ring which is neither left nor right algebraic over a division subring $E$, it is not clear that we can choose a $b \in D$ which is simultaneously neither left nor right algebraic over $E$.  We will finesse this issue below by adjoining an indeterminate and applying the following result.  
\begin{lem}
\label{lem: both sides}
Let $D$ be a division ring with division subring $E$.  Suppose that $D$ is neither left nor right algebraic over $E$. Then there is $b(t) \in D(t)$ such that $b(t)$ is neither left nor right algebraic over $E(t)$.
\end{lem}
\begin{proof}
Pick $u,v \in D$ such that $\sum_{i \geq 0} E u^i$ and $\sum_{i \geq 0} v^i E$ are direct.  Let $b(t):=u+tv$ in $D(t)$.  Suppose that $\sum_{i \geq 0} E(t)b(t)^i$ is not direct.  Clearing denominators on the left gives $\sum_{i=0}^n e_i(t) b(t)^i = 0$ where $e_0,\ldots ,e_n \in E[t]$ are not all zero.  We may assume that some $e_i(t)$ has nonzero constant term.  Then we can specialize at $t=0$ and we get that  $u$ is left algebraic over  $E$, a contradiction.  Thus $b(t)$ is not left algebraic over $E(t)$.

Next, suppose that  $\sum_{j=0}^n b(t)^j e_j(t) = 0$  with $e_j(t)$ not all zero in $E[t]$.  Let $m := \max_{j = 0}^n (\operatorname{deg}(e_j)+j)$.  By considering the coefficient of $t^m$ on both sides, we see that $v$ is right algebraic over $E$, again a contradiction.  Thus $b(t)$ is not right algebraic over $E(t)$ either.
\end{proof}

Ideally, one would like to be able to say that if one has a chain of nested division rings
$$D_m\subseteq D_{m-1}\subseteq \cdots \subseteq D_0$$ with each $D_i$ left algebraic over $D_{i+1}$, then $D_0$ is necessarily left algebraic over $D_m$.  This may indeed be the case, but it is not obvious in general.  The next lemma shows that a weaker version of this desired property holds.

In the following proof and in several results later, it will be convenient to use power series.  Suppose that $D$ is a division ring, and let $t$ be an indeterminate.  Then elements in $D(t)$ can be identified with elements of the Laurent power series ring
\[
D((t)) := \{ \sum_{i=n}^{\infty} a_i t^i \,| \, n \in \mb{Z}, a_i \in D \}.
\]
  Namely, $D[t]$ naturally embeds in this ring, and since it is easy to prove directly that $D((t))$ is also a division ring, $D(t)$ also embeds in $D((t))$.  In particular, of course, $(1-at)^{-1} = 1 + at + a^2t^2 + \dots$ for any $a \in D$. This embedding also played a prominent role in Makar-Limanov's original proof that the quotient division ring of the Weyl algebra contains a free subalgebra; we refer it as the canonical embedding.
\begin{lem} Let $D_2\subseteq D_1\subseteq D_0$ be a chain of division rings.  If $D_0(t)$ is left algebraic over $D_1(t)$ and $D_1(t)$ is left algebraic over $D_2(t)$, then $D_0$ is left algebraic over $D_2$.\label{lem: chain}
\end{lem}
\begin{proof}  Let $a\in D_0$. Since $D_0(t)$ is left algebraic over $D_1(t)$, the sum
$$\sum_{i\ge 0} D_1(t)(1-at)^{-i}$$ is not direct.  By clearing left denominators in a dependency relation we get a relation $\sum_{i  = 0}^m  r_i(t)(1-at)^{-i} = 0$ with $r_i(t) \in D_1[t]$, $r_m(t) \neq 0$.   Multiplying on the right by $(1-at)^{m -1}$ we see that $0 \neq p(t) = r_m(t)$ satisfies $p(t)(1-at)^{-1} \in D_0[t]$.

Similarly, since $D_1(t)$ is left algebraic over $D_2(t)$, we have that
$$\sum_{i\ge 0} D_2(t)p(t)^{-i}$$ is not direct and hence there is some nonzero $q(t)\in D_2[t]$ such that $q(t)p(t)^{-1}\in D_1[t]$.  Take $s_1(t):=q(t)p(t)^{-1}\in D_1[t]$.  Write
$$q(t)=q_0+\cdots + q_m t^m$$ with $q_0,\ldots ,q_m\in D_2$ and $q_m\neq 0$.  Then $q(t)(1-at)^{-1} = q(t)p(t)^{-1} p(t) (1-at)^{-1}\in D_0[t]$ and so by using the embedding of $D_0(t)$ into $D_0((t))$, we see that $$q_0 a^n+q_1 a^{n-1}+\cdots + q_m a^{n-m}=0$$ for all sufficiently large $n$.  It follows that $$q_0 a^m+\cdots + q_m = 0$$ with $q_m \neq 0$, and so $a$ is left algebraic over $D_2$.
\end{proof}

\section{The main criterion}
\label{sec: main} In this section, we prove our most general criterion for the existence of a free subalgebra in 2 generators in a division ring.  It depends crucially on Proposition~\ref{lem: Ro}, and thus on the uncountability of the base field $k$.

We begin with a few technical lemmas.  Throughout the next two lemmas and the following theorem, we will use the following hypothesis:
\begin{hypothesis}
\label{hyp: usual}
Let $D$ be a division algebra over a field $k$ with $k$-automorphism $\sigma: D \to D$, and let $E := \{ a \in D \,| \, \sigma(a) = a \}$ be the division subring of $\sigma$-fixed elements.
\end{hypothesis}

\begin{lem} Assume Hypothesis~\ref{hyp: usual} and let $b\in D$.   Suppose that there exists a nonnegative integer $m$ and $q_0,\ldots ,q_m\in D$, not all zero, such that for some $n \geq m$ and all $j \gg 0$ we have $\sum_{i=0}^m \sigma^j(b^{n-i}) q_i  = 0$.  Then $b$ is right algebraic over $E$.
\label{lem: sigma}
\end{lem}

 \begin{proof} Pick a non-trivial relation of the form $$\sum_{i=0}^m \sigma^j(b^{n-i}) q_i= 0$$ that holds for some $n\ge m$ and all $j \gg 0$.   Choose a relation with $m$ minimal among all such relations.  If $m = 0$, then since $D$ is a domain we obtain $b = 0$ and the result is clear, so assume that $m \geq 1$.  It is no loss of generality to assume that $q_m=1$.  Applying $\sigma$, we see that $$\sum_{i=0}^m  \sigma^j(b^{n-i}) \sigma(q_i) = 0$$ for all $j \gg 0$, and hence subtracting we have $$\sum_{i=0}^{m-1} \sigma^j(b^{n-i}) (q_i-\sigma(q_i)) = 0$$ for all $j \gg 0$.  By minimality of $m$ we have that $q_i=\sigma(q_i)$, in other words $q_i \in E$,  for $0\le i\le m$.   Then applying  $\sigma^{-j}$ to the original relation for $j \gg 0$ we get $$\sum_{i=0}^m (b^{n-i}) q_i= 0,$$ which shows that $b$ is right algebraic over $E$.
\end{proof}

\begin{lem} Assume Hypothesis~\ref{hyp: usual} and let $b\in D$.  If there exist $r_0,\ldots ,r_p\in D$ with $r_p\not =0$ such that $\sum_{j=0}^{p} r_j \sigma^j(b^n)=0$ for all $n \gg 0$, then $b$ is right algebraic over $E$.
\label{lem: not direct}
\end{lem}
\begin{proof} Using the right Ore property of $D[t]$,  we may choose a polynomial $0 \neq Q(t) = q_0+q_1t+\cdots + q_m t^m$ in $D[t]$ with $q_m = 1$, such that for $0\le j < p$ we have 
 $$(1-\sigma^j(b)t)^{-1}Q(t)\in D[t].$$  We claim now that necessarily $(1-\sigma^j(b)t)^{-1}Q(t)\in D[t]$ for all $j \geq 0$.  The proof is by induction.   Suppose that $(1-\sigma^j(b)t)^{-1}Q(t)\in D[t]$ for all $0 \leq j < s$, where $s \geq p$.  We note that by applying $\sigma^{s-p}$ to the relation $\sum_{j=0}^{p} r_j \sigma^j(b^n)=0$, we see that  $\sum_{j=0}^{p} \sigma^{s-p}( r_j) \sigma^{s-p+j}(b^n)=0$ for all $n \gg 0$.
 
Equivalently, $$\sum_{j=0}^p \sigma^{s-p}( r_j) (1-\sigma^{s-p+j}(b)t)^{-1} = \sum_{n \geq 0} \bigg(\sum_{j=0}^p \sigma^{s-p}( r_j) \sigma^{s-p+j}(b^n) \bigg) t^n \in D[t].$$
Right multiplying by $Q(t)$, we see that $$\sum_{j=0}^p \sigma^{s-p}( r_j) (1-\sigma^{s-p+j}(b)t)^{-1} Q(t) \in D[t],$$
and since $(1-\sigma^i(b)t)^{-1}Q(t) \in D[t]$ for all $i < s$, we conclude that
$$\sigma^{s-p}(r_p)(1-\sigma^s(b)t)^{-1}Q(t)\in D[t],$$
and thus $(1-\sigma^s(b)t)^{-1}Q(t)\in D[t]$.  This completes the induction step and thus the  proof of the claim.

If we now consider the embedding of $D(t)$ in $D((t))$, since $(1- \sigma^j(b)t)^{-1} =
\sum_{i \geq 0} \sigma^j(b^i) t^i$, looking at the $n$th coefficient of $(1-\sigma^j(b)t)^{-1}Q(t)$ we see that
$$\sum_{i=0}^m \sigma^j(b^{n-i}) q_i = 0$$ for all $n \gg m$ and all $j\in \mathbb{N}$.  The result follows by Lemma \ref{lem: sigma}.
\end{proof}

With the lemmas above in hand, we are now ready to prove our main result, which relates the existence of free subalgebras of $D(x; \sigma)$ to algebraicity of elements over the $\sigma$-fixed field $E$.

\begin{thm} Assume Hypothesis~\ref{hyp: usual}.  Let $k$ be an uncountable field.  If $\operatorname{char} k =  p > 0$, then assume in addition that there is no element $u \in D$ such that $\sigma(u) = u + 1$.  If there exists $b \in D$ which is neither left nor right algebraic over $E$, then $D(x;\sigma)$ contains a free $k$-subalgebra on two generators.
\label{lem: delt}
\label{thm: delt}
\end{thm}

\begin{proof} We pick a $b$ which is non-algebraic over $E$ on both sides, so both $\sum_{j=0}^{\infty} b^jE$ and $\sum_{j=0}^{\infty} Eb^j$ are direct.  Suppose that $D(x;\sigma)$ does not contain a free algebra on two generators.   As $k$ is uncountable, $D(x; \sigma)(t)  = D(t)(x;\sigma)$ also does not contain a free $k$-algebra on two generators by Proposition~\ref{lem: Ro}, where we extend $\sigma$ to $D(t)$ by declaring that $\sigma(t)=t$.  Note that $E(t)$ is the division subalgebra of $D(t)$ of $\sigma$-fixed elements, by Lemma~\ref{dumb-lem}(1).  In particular, since $b \not \in E$, $(1-bt)^{-1}\not\in E(t)$.  If $\operatorname{char} k =  p > 0$, we note that that since $D$ contains no element $u$ with $\sigma(u) = u + 1$, $D(t)$ also contains no such element by Lemma~\ref{lem: weyl}.  Theorem \ref{main-criterion-thm} now shows that there exist $u(t)\in D(t)$ and $e(t)\in E(t)$ such that
$$u(t)-\sigma(u(t)) \ = \ (1-bt)^{-1} + e(t).$$
Using the canonical embedding of $D(t)$ into $D((t))$, we can write
$$u(t)=\sum_{j=-m}^{\infty} u_j t^j \qquad \text{and} \qquad e(t)=\sum_{j=-m}^{\infty} e_j t^j,$$
for some $m \in \mb{Z}$. Then looking at the coefficient of $t^n$ in the equation above for any $n \geq 0$, we see that
\begin{equation}
\label{eq: coeff}
u_n - \sigma(u_n) = b^n + e_n.
\end{equation}
Since $u(t)\in D(t)$, there exists some nonzero polynomial $s(t)=\sum_{i=0}^d s_i t^i\in D[t]$ such that $s(t)u(t)\in D[t]$.  Equivalently,
$$\sum_{i=0}^d s_i u_{n-i} = 0$$ for all $n \gg 0$.  We pick the smallest $m\in \{-1,0,1,2,\ldots\}$ such that there exists some finite-dimensional right $E$-vector subspace $W \subseteq D$, $t_0,\ldots ,t_m\in D$, and $r_0,\ldots, r_p\in D$ with $t_0,\ldots ,t_m,r_0,\ldots ,r_p$ not all zero, such that
\begin{equation}
\label{eq: imp}
\sum_{i=0}^m t_i u_{n-i} + \sum_{j=0}^p r_j\sigma^{-j}(b^n)\in W
\end{equation} for all $n$ sufficiently large.  (We take $m=-1$ to mean that the first sum is absent.)
Since $\sum_{i=0}^d s_i u_{n-i}=0$ for all sufficiently large $n$, we certainly have $m\le d$.

Our goal is to show that $m=-1$.  Suppose that $m\ge 0$.  We may assume $t_m=1$ by left-multiplying by an appropriate element of $D$, and then we may also assume that either $r_p\not =0$, or else $p=0$ and $r_0=0$. Note that $\delta := 1- \sigma$ is a $\sigma$-derivation, that is, that $\delta(xy) = \delta(x) \sigma(y) + x \delta(y)$ for all
$x, y \in D$.
Then if we apply $\delta$ to both sides of \eqref{eq: imp} and use also \eqref{eq: coeff}, we see that
$$\sum_{i=0}^{m-1}\delta(t_i) \sigma(u_{n-i})
+ \sum_{i=0}^m t_i (b^{n-i}+e_{n-i})\\
+ \sum_{j=0}^p \left(\delta(r_j)\sigma^{-j+1}(b^n) + r_j\sigma^{-j}(b^n) - r_j \sigma^{-j+1}(b^n)\right)$$
is in $\delta(W)$ for all $n \gg 0$.  Note that $\delta(W)$ is also finite-dimensional over $E$ on the right, as
$$\delta(\sum v_i E) = \sum \delta(v_i)E.$$
Thus applying $\sigma^{-1}$ we see for all $n \gg 0$ that
\begin{eqnarray*}
\sum_{i=0}^{m-1}\sigma^{-1}(\delta(t_i)) u_{n-i}
&+& \left( \sum_{i=0}^m \sigma^{-1}(t_ib^{-i}) \right) \sigma^{-1}(b^{n}) \\
&+& \sum_{j=0}^p \left(\sigma^{-1}(\delta(r_j))\sigma^{-j}(b^n) + \sigma^{-1}(r_j)\sigma^{-j-1}(b^n) - \sigma^{-1}(r_j) \sigma^{-j}(b^n)\right)\end{eqnarray*} is in $\sigma^{-1}(\delta(W)) + \sum_{i=0}^m \sigma^{-1}(t_i)E$, which is also finite-dimensional as a right $E$-vector space.  By minimality of $m$, this relation must be trivial.  This means that $\delta(t_i)=0$ for $0\le i<m$ and so $t_0,\ldots ,t_{m-1}\in E$; also, the coefficient of $\sigma^{-j}(b^n)$ must be $0$ for $0\le j\le p+1$.  Note that if $p>0$ and $r_p\not =0$, then the coefficient of $\sigma^{-p-1}(b^n)$ is $\sigma^{-1}(r_p)\neq 0$, a contradiction.  Hence $p=0$.  Then the coefficient of $b^n$ is just $\sigma^{-1}(\delta(r_0))-\sigma^{-1}(r_0)$.  Thus $\delta(r_0)=r_0$ and so $r_0=0$.  Finally, looking at the coefficient of $\sigma^{-1}(b^n)$ we see that
$$\sum_{i=0}^m t_i b^{-i}=0.$$
Multiplying by $b^m$ on the right we see that $\sum_{i \geq 0} Eb^i$ is not direct, a contradiction.  Thus $m=-1$.

We now have  a finite-dimensional right $E$-vector space $W$ and $r_0,\ldots ,r_p$ in $D$ not all zero such that \begin{equation} \label{eq: imp2}
\sum_{i=0}^p r_i\sigma^{-i}(b^n)\in W
\end{equation} for all $n$ sufficiently large.  We can now assume that $r_p \neq 0$ and we pick the right $E$-vector space $W$ of smallest dimension for which there exists such a relation.  We claim that $W=(0)$.  If not, then by left-multiplying our relation by an appropriate nonzero element, we may assume that $1\in W$.  Let $1=a_1,\ldots , a_{\ell}$ be a basis for $W$ as a right $E$-vector space, where $\ell \geq 1$.  Then
$$\sum_{i=0}^p r_i\sigma^{-i}(b^n) \ \in  \  \sum_{i=1}^{\ell} a_i E$$
for all $n \gg 0$. If we apply $\delta\circ \sigma^{-1} = \sigma^{-1} - 1$ to both sides of this relation, we see there are constants $r_0',\ldots ,r'_{p+1} \in D$ such that $$\sum_{i=0}^{p+1} r_i'\sigma^{-i}(b^n) \in \sum_{i=2}^{\ell} \delta(\sigma^{-1}(a_i))E$$ for all $n \gg 0$.  Since $r'_{p+1} = \sigma^{-1}(r_p) \neq 0$, this contradicts the minimality of the dimension of $W$.  Hence $W=(0)$.

Thus we see that there exist $r_0,\ldots ,r_p\in D$ not all zero such that
$$\sum_{j=0}^p r_j \sigma^j (b^n) = 0$$ for all $n$ sufficiently large.  Thus by Lemma \ref{lem: not direct}, $b$ is right algebraic over $E$, a contradiction.  The result follows.
\end{proof}

The fact that the result above obtains free subalgebras only of division algebras of the form $D(x; \sigma)$ is not as special as it might first appear.  In fact, we have the following immediate corollary.
\begin{cor}
Let $D$ be a division algebra over an uncountable field $k$.  Suppose there are nonzero $a,b\in D$ such that $b$ is neither left nor right algebraic over $E$, where $E := C(a; D)$ is the centralizer of $a$ in $D$.    If the characteristic of $k$ is positive, assume in addition that $D$ does not contain an element $u$ such that $aua^{-1} = u + 1$.  Then $D$ contains a free $k$-subalgebra on two generators.
\label{cor: xxx}
\end{cor}
\begin{proof}  Let $\sigma:D\to D$ denote the inner automorphism given by conjugation by $a$.  Then $E$ is the fixed division subring of $\sigma$.  Note that $D(x;\sigma) \cong D(t)$, by identifying $x$ with $at$.  If $\operatorname{char} p > 0$, then by hypothesis there is no $u \in D$ such that $\sigma(u) = u + 1$.  It follows from Theorem~\ref{thm: delt} that $D(x;\sigma)\cong D(t)$ contains a free $k$-algebra on two generators.  Proposition~\ref{lem: Ro} then gives that $D$ contains a free $k$-algebra on two generators.
\end{proof}

We also do not need to worry about the need to pick an element which is simultaneously non-algebraic over $E$ on both sides in Corollary~\ref{cor: xxx}, as follows.  
\begin{cor}
\label{cor: xxx2}
Let $D$ be a division algebra over an uncountable field $k$.  Suppose there is a nonzero $a \in D$ such that $D$ is neither left nor right algebraic over $E := C(a; D)$.    If the characteristic of $k$ is positive, assume in addition that $D$  does not contain an element $u$ such that $aua^{-1} = u + 1$.  Then $D$ contains a free $k$-subalgebra on two generators.
\end{cor}

\begin{proof} Since $D$ is not algebraic over $E$ on either side, Lemma~\ref{lem: both sides} shows that we may pick a single element $b(t) \in D(t)$ which is not algebraic over $E(t)$ on either side.  If $k$ has characteristic $p > 0$, then letting $\sigma: D(t) \to D(t)$ be conjugation by $a$, Lemma~\ref{lem: weyl} shows that $D(t)$ also contains no element $u$ such that $\sigma(u) = u + 1$.  Since $E(t) = C(a; D(t))$ by Lemma~\ref{dumb-lem}(2), the previous corollary applied to $D(t)$ shows that $D(t)$ contains a free $k$-algebra on two generators.  Thus $D$ does also by Proposition~\ref{lem: Ro}.
\end{proof}

We mention one more immediate application.
\begin{cor} Let $D$ be a division ring of characteristic $0$ that is not algebraic over its centre and let $\sigma$ be an automorphism of $D$.  If
$$k:=\{x\in D \,| \, \sigma(x)=x\}$$
is uncountable and contained in the centre of $D$, then $D(x;\sigma)$ contains a free $k$-algebra on two generators. \label{cor: 2}
\end{cor}
\begin{proof}  By assumption, in the notation of Hypothesis~\ref{hyp: usual} we have $E=k$, and $D$ is not algebraic over $k$.  Thus we can find $b\in D$ such that $\sum_{j\ge 0} Eb^j = \sum_{j \geq 0} b^j E$ is direct.  It follows that $D(x;\sigma)$ contains a free $k$-algebra on two generators by Theorem \ref{lem: delt} in the case that the characteristic of $k$ is zero.
\end{proof}

\noindent  We note that the preceding corollary recovers yet again the special case of \cite[Theorem 1.1]{BR} given in Theorem~\ref{thm: BR} above, since if $D = K$ is a field then it is an algebraic extension of its subfield $k$ of $\sigma$-fixed elements if and only if every element of $K$ is on a finite $\sigma$-orbit.

\section{The quotient rings of noetherian domains} In this short section, we improve the criterion of the previous section even further for $D := Q(A)$, where $A$ is is a countably generated noetherian algebra over an algebraically closed uncountable field of characteristic $0$.

\begin{prop} Let $k$ be an uncountable field of characteristic $0$ and let $D$ be a division algebra over $k$.  Suppose that $D$ has a maximal subfield $K$ that is a finitely generated extension of $k$,  and such that $D$ is not left algebraic over $K$.  Then $D$ contains a free $k$-algebra on two generators.  
\label{maxsub}

\end{prop}
\begin{proof}  We will prove this by induction on the size of a minimal generating set for $K$ as an extension of $k$.   By Lemma~\ref{lem: fix}, the hypothesis that $D$ is not left algebraic over $K$ implies that $D$ is neither left nor right algebraic over $K$.  We note that $K\neq k$ since $D\neq K$.  Let $a_1,\ldots ,a_d\in K$ be a minimal set of generators of $K$ as an extension of $k$.

Suppose that $D$ does not contain a free $k$-algebra on two generators.  Suppose that $d = 1$, so that $K=k(a_1)$ with $a_1 \not \in k$.   Note that since $K$ is a maximal subfield, it is its own centralizer and thus in fact $K = C(a_1; D)$, the centralizer of $a_1$ in $D$.    Since $D$ does not contain a free $k$-algebra on two generators, $D$ must be either left or right algebraic over $K$, by Corollary \ref{cor: xxx2}; but this is a contradiction.  Thus the result holds for the case $d=1$.

Now assume $d \ge 2$ and that the result holds for smaller $d$.  As above, suppose that $D$ does not contain a free $k$-algebra on two generators.  Let $D_1 :=C(a_1 ; D)$.  Since $a_1\in D$, we have $D_1(t)=C(a_1 ; D(t))$ by Lemma~\ref{dumb-lem}(2).  Then $k(t)(a_1)\subseteq Z(D_1(t))$ and $K(t)$ is a maximal subfield of $D_1(t)$ (by Lemma~\ref{dumb-lem}(2) again) that is generated by $d-1$ elements over $k(t)(a_1)$. By Proposition~\ref{lem: Ro}, $D(t)$ cannot contain a free $k$-algebra on two generators, and hence $D_1(t)$ cannot contain a free $k(t)(a_1)$-algebra on two generators.  By the inductive hypothesis, $D_1(t)$ is left algebraic over $K(t)$, and hence $D_1(t)$ is also right algebraic over $K(t)$ by Lemma~\ref{lem: fix}.  But $D(t)$ must also be either left or right algebraic over $D_1(t)$, or else Corollary~\ref{cor: xxx2} implies that $D(t)$ contains a free algebra.  Hence $D$ must be either left or right algebraic over $K$ by Lemma~\ref{lem: chain} (or its right-handed version), again a contradiction.  The result follows.
\end{proof}

\begin{cor}

\label{cor-mainthm2}
Let $A$ be a noetherian domain which is a countably generated algebra over an algebraically closed uncountable base field $k$ of characteristic $0$.  Then $Q(A)$ either contains a free $k$-algebra on two generators or it is left algebraic over every maximal subfield.
\end{cor}
\begin{proof} The hypotheses on $A$ guarantee that all maximal subfields of $Q(A)$ are finitely generated over $k$ \cite[Corollary 1.3]{Bell3}.  If $Q(A)$ has some maximal subfield $K$ over which it is not left algebraic then $Q(A)$ contains a free algebra by Proposition \ref{maxsub}.
\end{proof}

\section{Domains of GK-dimension $2$}

\label{GK2}
In this section, we consider finitely generated $k$-algebras $A$ with GK-dimension strictly less than $3$ that are domains.  It is conjectured that such an algebra must have GK-dimension either $0$, $1$, or $2$, but currently it is unknown if algebras of dimenson between $2$ and $3$ are possible.  We show that if the base field is algebraically closed, then either $A$ is PI or $Q(A)$ contains a free $k$-algebra on two generators.   We note that our proof uses ideas of Bergman.  The main idea of the proof of Bergman's gap theorem has many other applications; for example, it can be used to show that the Kurosh conjecture holds for algebras of GK-dimension $1$, and also to give an easy proof of the Small-Warfield theorem for prime Goldie algebras of GK-dimension $1$.  The following lemma is another application of the basic idea of Bergman's proof.

 \begin{lem} Let $\Sigma := \{x_1,\ldots, x_e\}$ be a finite alphabet, let $W\subseteq \Sigma^*$ be a subset of finite words over the alphabet $\Sigma$, and let $f(n)$ denote the number of words of length $n$ in $W$ for each natural number $n$.  Suppose that $W$ has the following properties:
 \begin{enumerate}\item $W$ is infinite;
 \item if $w\in W$ and $v$ is a subword of $w$ then $v\in W$;
 \item there is a natural number $d$ such that $f(d)\le d$.
\end{enumerate}
Then there is some word $u\in W$ such that $u^n\in W$ for every natural number $n$.
\label{lem: words}
\end{lem}
 \begin{proof}
 Let $W_n$ be the set of words in $W$ which have length $n$.  For each $m \leq n$ there is a truncation map $\phi_{n,m}: W_n \to W_m$ which sends a word to its first $m$ letters.  Let $V_n \subseteq W_n$ be defined by $V_n := \bigcap_{p \geq n} \phi_{p,n}(W_p)$; in other words, $V := \bigcup V_n$ is the set of words in $W$ which have extensions of arbitrarily long length to other words in $W$.  Since $\phi_{p+1,n}(W_{p+1}) \subseteq \phi_{p,n}(W_p)$ for all $p \geq n$, we have $V_n = \phi_{p,n}(W_p)$ for all $p \gg 0$.  Thus $V_n$ is nonempty for all $n$, since $W$ is infinite.  Notice that $V$ is also closed under taking subwords; that is, if $w\in V$ and $v$ is a subword of $w$, then $v\in V$.   Moreover, the truncation map $\phi_{n,m}$ restricts to a surjective function $V_n \to V_m$ for any $m \leq n$, since $\phi_{n,m}(V_n) = \phi_{n,m} \phi_{p,n}(W_p) = \phi_{p,m}(W_p) = V_m$ for all $p \gg 0$.

 Let $g(n) := |V_n|$.  Since $g(n) \le f(n)$, there is some natural number $d$ such that $g(d)\le d$.  Notice that $g(i) \le g(i+1)$ for all $i$, by the surjectivity of the truncation map.  It follows that there is some $i$ such that $g(i)=g(i+1)$, for if this were not the case, then we would have $g(j+1)\ge g(j)+1$ for every natural number $j$, and so $g(d)\ge d-1+g(1)$.  Since $g(d)\le d$, we see that $g(1)=1$.  But $V$ is closed under taking subwords and so every element of $V$ would then be a power of some letter $x_i \in \Sigma$, and in this case $g(i)=g(i+1) =1$ for every natural number $i$, a contradiction.

Choose $i \in \mb{N}$ such that $m:=g(i)=g(i+1)$.  Let $V_i = \{ w_1,\ldots ,w_m \}$.  By definition, for each $j \in \{1,\ldots ,m\}$ there exists a letter $x_{a_j}\in \Sigma$ such that $w_j x_{a_j} \in V$, and moreover since $g(i+1)=g(i)$, we see that $x_{a_j}$ is uniquely determined.  For each $j\in \{1,2,\ldots ,m\}$, we write $$w_jx_{a_j}=x_{b_j} u_j$$ for some word $u_j$ of length $i$.  Since $V$ is closed under subwords, we see that $u_1,\ldots ,u_m\in V_i$.  But this means that there is a unique way of extending each $w_j$ to a word of length $i+2$ by adding two letters to the end.  In particular, we have $g(i+2)=g(i)$.  Continuing inductively, we see that $g(n)=g(i)$, and that every word in $V_i$ has a unique extension to a word in $V_n$, for every $n \geq i$.

Pick a word $v$ of length $i(m+1)$ in $V$.  Then as $V$ is closed under taking subwords, we necessarily have $$v=w_{j_1}\cdots w_{j_{m+1}}$$ for some $j_1,\ldots ,j_{m+1}\in \{1,\ldots ,m\}$.  Then there are $a$ and $b$ with $a<b$ such that $j_a=j_b$.  Let $u:=w_{j_a}w_{j_a+1}\cdots w_{j_{b-1}}$.  Since for $a \leq k \leq b-2$, $w_{j_k} w_{j_{k+1}}$ is the unique extension of $w_{j_k}$ to a word in $V_{2i}$, and $w_{j_{b-1}}w_{j_b} = w_{j_{b-1}}w_{j_a}$ is the unique extension of $w_{j_{b-1}}$ to a word in $V_{2i}$, it is clear that for each $n \geq 1$ the unique word in $V_{i(b-a)n}$ extending $u$ must be $u^n$.  In particular, $u^n\in W$ for every natural number $n$.
\end{proof}

We need the following technical result about hypothetical division algebras which are algebraic over a subfield, but are not locally PI, which will help us show that such algebras cannot have too small growth.
\begin{prop} Let $k$ be a field and let $D$ be a division $k$-algebra.  Suppose that $D$ is either left or right algebraic over a subfield $L$ that is finitely generated as an extension of $k$.  Let $V$ be a finite-dimensional $k$-vector subspace of $D$ which  contains $1$ and which generates a non-PI subalgebra of $D$.  Then there exists a finite-dimensional $k$-vector space $U \subseteq D$ such that $V \subseteq U$ and $LVL = UL = LU$.   Any such $U$ satisfies ${\rm dim}_L U^nL \ge {n+2\choose 2}$ as a right $L$-vector space, for all $n \geq 0$.
\label{prop: growth}
\end{prop}
\begin{proof} Lemma~\ref{lem: fix} shows that $D$ is algebraic over $L$ on both sides, and that $LVL$ is a finite-dimensional $L$-vector space on both sides.  Any finite-dimensional $k$-space $U$ containing $V$ and containing both a right and left $L$-spanning set for $LVL$ will satisfy $LVL = UL = LU$; in particular, such a $U$ exists.  Pick a basis $\mathcal{B} := \{b_1,\ldots , b_e=1\}$ for $U$ as a $k$-vector space and let $\Sigma :=\{x_1,\ldots ,x_e\}$ be a set of indeterminates.  We put a degree lexicographic ordering on the words in the free monoid $\Sigma^*$ by declaring that $x_1>x_2>\cdots > x_e.$  We note that we have a map of monoids $\phi : \Sigma^* \to D$ given by $\phi(x_i)=b_i$.

We define a subset $W$ of $\Sigma^*$ as follows: $$ W := \{ w \in \Sigma^* \,| \, \phi(w) \not \in \sum_{u<w} \phi(u) L \}.$$  Clearly $\phi(W)$ is a right $L$-basis for $\sum_{n \geq 0} U^n L$.  We note that $W$ is infinite; if not, we have that $\phi(W)\subseteq U^d$ for some natural number $d$, and so $\sum_{n \geq 0} U^n L =  U^d L$.  But since $V$ generates a non-PI algebra $A$ as a $k$-algebra, and $V \subseteq U$, $U$ also generates a non-PI algebra.  On the other hand, $\sum_{n \geq 0} U^n L$ is a ring (since $UL = LU)$ which is a finite-dimensional right $L$-vector space and so must satisfy a PI \cite[Theorem 2]{PS}, a contradiction.

We next observe that $W$ is closed under taking subwords.  To see this, let $w\in W$ and suppose that $w=w_1vw_2$ for some words $w_1,w_2,v\in \Sigma^*$.  Suppose that $v\not \in W$.  Then we have $$\phi(v) \in \sum_{u<v} \phi(u)L.$$ Thus
$$\phi(w)=\phi(w_1)\phi(v)\phi(w_2) \in \sum_{u<v} \phi(w_1)\phi(u)L\phi(w_2).$$
Let $m$ denote the length of $w_2$.  Then
$$\sum_{u<v} \phi(w_1)\phi(u)L\phi(w_2) \subseteq \sum_{u<v} \phi(w_1)\phi(u)LU^m
 \ =  \  \sum_{u<v} \phi(w_1)\phi(u)U^m L.$$
 Furthermore, if $u<v$ then $w_1uy<w_1vw_2$ for every word $y$ of length at most $m$, and so $\sum_{u<v} \phi(w_1)\phi(u)U^m L \subseteq \sum_{u<w} \phi(u)L$. But this means that
$$\phi(w)\in \sum_{u<w} \phi(u)L,$$ a contradiction to $w \in W$.  Thus $v \in W$ and $W$ is closed under taking subwords.

Let $W_n$ denote the subset of $W$ consisting of those words of length $n$, and let $f(n) := |W_n|$.  By construction, ${\rm dim}_L(U^nL) = \sum_{i = 0}^n f(i)$.  If $f(d)> d$ for each $d$, we get
$${\rm dim}_L(U^nL) \ge {n+2\choose 2}$$ as required.
Thus we may assume that $f(d)\le d$ for some $d$.  By Lemma~\ref{lem: words}, there exists some $u\in W$ such that $u^n\in W$ for every $n$.  It follows that $\sum_{i \geq 0} \phi(u)^i L$ is direct, which contradicts the fact that $D$ is right algebraic over $L$.  The result follows.
\end{proof}

\begin{cor} Let $k$ be an uncountable field and let $A$ be a finitely generated $k$-algebra that is a non-PI domain of GK-dimension strictly less than $3$.  If $A$ is not algebraic over $k$, then for any $x \in A$ which is transcendental over $k$ we have that $Q(A)$ is not algebraic over $k(x)$ on either side.
\label{prop: set}
\end{cor}
\begin{proof}  Pick $x\in A\setminus k$ which is transcendental over $k$, and let $V \subseteq A$ be a finite-dimensional $k$-vector space that contains $1$ and $x$ and generates $A$ as a $k$-algebra.  Let $L :=k(x)$.  Assume that $Q(A)$ is either left or  right algebraic over $L$.  Then by Proposition \ref{prop: growth}, we have that
$${\rm dim}_L U^nL \ge {n+2\choose 2}$$
for all $n \geq 0$, where $U \supseteq V$ is any finite-dimensional $k$-space such that $UL = LU = LVL$.  We claim that we can take $U \subseteq A$.  Note that we can take for $U$ the $k$-span of a basis for $V$, a left $L$-basis for $LVL$, and a right $L$-basis for $LVL$.   The left $L$-basis for $LVL$ can be taken in $VL$; moreover, right multiplication of the elements of this basis by a nonzero scalar in $L$ keeps it a left $L$-basis, so by clearing denominators it can be taken in $Vk[x] \subseteq A$.  Similarly, the right $L$-basis for $LVL$ can be taken in $k[x]V$, proving the claim.

Now if $\{v_1, \dots, v_m \} \subseteq U^n$ is a right $L$-basis of $U^nL$, where $m  \geq {n+2\choose 2}$,  then we have $${\rm dim}_k(U^{2n}) \ge {\rm dim}_k \left(\sum_{i\le n} U^n x^i \right)\ge {\rm dim}_k \left(\sum_{j=1}^m v_j \sum_{i\le n}  k x^i \right) \ge {n+2\choose 2}n.$$  But this contradicts the fact that $A$ has GK-dimension strictly less than $3$.  Consequently, $Q(A)$ cannot be either left or right algebraic over $k(x)$.
\end{proof}

We are now ready to prove the main result of this section.  First we give a result that does not require the base field to be algebraically closed.
\begin{thm} Let $k$ be an uncountable field and let $A$ be a finitely generated $k$-algebra that is a domain of GK-dimension strictly less than $3$, such that $A$ is not algebraic over $k$; fix $x \in A$ which is transcendental over $k$.  If $\operatorname{char} k = p > 0$, assume in addition that $D: = Q(A)$ does not contain an element $u$ such that $xux^{-1} = u + 1$.  If $A$ does not satisfy a polynomial identity, then $Q(A)$ contains a free $k$-algebra on two generators.
\end{thm}
\begin{proof} We assume that $A$ does not satisfy a polynomial identity.  Let $L := k(x)$.  By Corollary~\ref{prop: set},  $D = Q(A)$ is not algebraic over $k(x)$ on either side.  Let $E := C(x; D)$ be the centralizer of $x$. We add an indeterminate $t$ and consider
\[
k(t) \subseteq k(x, t) \subseteq E(t) = C(x; D(t)) \subseteq D(t),
\]
where we have used Lemma~\ref{dumb-lem}(2).  Now $A' := A \otimes_k k(t)$ is a finitely generated $k(t)$-algebra which is a domain with ${\rm GKdim}_{k(t)} A' = {\rm GKdim}_k A < 3$, and $D(t) = Q(A')$.  Suppose that $E(t)$ is not algebraic over $k(x, t)$, and pick $y \in E(t)$ which is transcendental over $k(x, t)$.  We see by \cite[Theorem 1.2]{Bell} that since the chain $k(t) \subseteq k(x, t) \subseteq k(x, y, t)$ already has two infinite-dimensional extensions, we must have that $Q(A')$ is finite-dimensional over the field $k(x, y, t)$ and thus is PI.  Then $A$ is PI, a contradiction.   Thus $E(t)$ is algebraic over $k(x, t)$.

If $D(t)$ is algebraic over $E(t)$ on the left, then we get that $D$ is left algebraic over $k(x)$ by Lemma~\ref{lem: chain}, and this is a contradiction.  Similarly, $D(t)$ is not algebraic over $E(t)$ on the right, using the right-handed version of Lemma~\ref{lem: chain}.  If $k$ has characteristic $p$, then $D(t)$ also does not have an element $u$ such that $xux^{-1} = u +1$, using Lemma~\ref{lem: weyl}.  Now Corollary~\ref{cor: xxx2} shows that $D(t)$ contains a free $k$-subalgebra on two generators.   So $D$ contains a free algebra by Proposition~\ref{lem: Ro}.
\end{proof}

We close with a version of the theorem over an algebraically closed field, in which case we can give a very clean statement.
\begin{cor}
\label{cor-mainthm3} Let $k$ be an algebraically closed uncountable field and let $A$ be a finitely generated $k$-algebra that is a domain of GK-dimension strictly less than $3$.  If $A$ does not satisfy a polynomial identity, then $Q(A)$ contains a free $k$-algebra on two generators.
\end{cor}
\begin{proof}  If $k$ has characteristic $0$, then this follows immediately from the preceding theorem.  If $k$ has characteristic $p > 0$, then choose any $x \in A \setminus k$, which will necessarily be transcendental over $k$; then the theorem will apply if we show that there is no $u \in D := Q(A)$ such that $xux^{-1} = u + 1$.  Suppose instead that there is such a $u$.  Then $xz = zx + 1$, where $z = ux^{-1}$.  Thus $R := k \langle x, z \rangle \subseteq D$ is a domain which is a homomorphic image of the Weyl algebra.  If $R$ is not isomorphic to the Weyl algebra, then $R$ has GK-dimension $1$, since $R$ contains $k[x]$.  But a domain of GK-dimension 1 which is an affine algebra over an algebraically closed field must be commutative, as an application of Tsen's theorem \cite[pp. 278--279]{Coh}.  Since a commutative ring cannot have elements satisfying the relation $xz = zx + 1$, this is a contradiction.  Thus $R$ is isomorphic to the Weyl algebra.  Now by \cite[Theorem 1.3]{Bell}, $Q(A)$ is finite-dimensional as a left $Q(R)$-vector space.  But since $Q(R)$ is PI, $Q(A)$ is PI.  This is a contradiction since $A$ is not PI.
\end{proof}

\section*{Acknowledgments}
We thank Jia-feng L\"u for pointing out a problem with our original proof of Lemma 3.1.

\end{document}